\documentclass[10pt]{article}

\usepackage[hmargin=35mm,vmargin=40mm]{geometry}

\usepackage{amsmath}              
\usepackage{amssymb}              
\usepackage{amsthm}               
\usepackage{enumerate}            
\usepackage{graphicx}             
\usepackage{mathrsfs}             
\usepackage[tight]{subfigure}     
\usepackage{url}                  

\newtheorem{theorem}{Theorem}[section]
\newtheorem{lemma}[theorem]{Lemma}

\newtheorem{algorithm}[theorem]{Algorithm}

\theoremstyle{definition}
\newtheorem*{defn}{Definition}

\newcommand{\bdry}{\partial}
\newcommand{\cc}[1]{C(#1)}
\newcommand{\comp}[1]{\overline{K}}
\newcommand{\cplex}{\emph{ILOG CPLEX}}
\newcommand{\genus}[1]{g(#1)}
\newcommand{\knotinfo}{\emph{KnotInfo}}
\newcommand{\mobius}{M{\"o}bius}
\newcommand{\normaliz}{\emph{Normaliz}}
\newcommand{\R}{\mathbb{R}}
\newcommand{\regina}{\emph{Regina}}
\newcommand{\snappy}{\emph{SnapPy}}
\newcommand{\tri}{\mathcal{T}}
\newcommand{\vrep}[1]{\mathbf{v}(#1)}
\newcommand{\Z}{\mathbb{Z}}

\newcommand{\co}{\colon\thinspace}

\title{Computing the crosscap number of a knot \\
    using integer programming and normal surfaces}
\author{Benjamin A.~Burton and Melih Ozlen}
\date{March 4, 2012}

\begin{document}

\maketitle

\begin{abstract}
    The crosscap number of a knot is an invariant describing the
    non-orientable surface of smallest genus that the knot bounds.
    Unlike knot genus (its orientable counterpart), crosscap numbers
    are difficult to compute and no general algorithm is known.
    We present three methods for computing crosscap number that offer
    varying trade-offs between precision and speed:
    (i)~an algorithm based on Hilbert basis enumeration and
    (ii)~an algorithm based on exact integer programming, both of which
    either compute the solution precisely or reduce it to two possible values,
    and (iii)~a fast but limited precision integer programming algorithm
    that bounds the solution from above.

    The first two algorithms advance the theoretical state of the art,
    but remain intractable for practical use.
    The third algorithm is fast and effective, which we show in a
    practical setting by making significant improvements
    to the current knowledge of crosscap numbers in knot tables.
    Our integer programming framework is general, with the potential for
    further applications in computational geometry and topology.
%
\end{abstract}

%
%

\section{Introduction}

Knot invariants lie at the heart of computational knot theory.
Nevertheless, computing invariants can be challenging:
algorithms often require complex implementations and exponential running
times, and for some invariants no general algorithm is known.

In this paper we focus on invariants of knots in $\R^3$
that relate to 2-dimensional surfaces:
\emph{knot genus} and \emph{crosscap number}.  In essence,
these measure the simplest embedded orientable and non-orientable
surfaces respectively that have a single boundary curve following the knot.

Knot genus is well-studied: algorithms are known \cite{hass99-knotnp},
and precise values have been computed for all prime knots with $\leq 12$
crossings \cite{www-knotinfo-jun11}.
In contrast, although crosscap numbers can be computed for
special classes of knots
\cite{hirasawa06-2bridge, ichihara10-pretzel, teragaito04-crosscap},
no algorithm is known for computing them in general.
Of all 2977 non-trivial prime knots with $\leq 12$ crossings,
only 289 have crosscap numbers that are known precisely
\cite{www-knotinfo-jun11}.

The crosscap number displays unusual behaviours that knot genus does not,
and embodies different information
\cite{clark78-crosscaps,murakami95-crosscap}.
It is therefore desirable to compute crosscap numbers in a general
setting.  Here we develop three algorithms with varying
precision-to-speed trade-offs, which yield both theoretical and
practical advances.

Our algorithms do not guarantee to compute the crosscap number precisely
for every input (this remains an open problem), but they do come close.
The first two algorithms are significant theoretical advances:
they either compute the crosscap number precisely or reduce it to one of two
possible values.  The third algorithm is a significant practical achievement:
although it only outputs an upper bound on the crosscap number, its
strong practical performance combined with known lower bounds allows us to
make significant improvements to crosscap numbers in existing tables of knots.

The first algorithm, described in Section~\ref{s-hilbert}, uses Haken's
normal surface theory \cite{haken61-knot} to reduce the
computation of crosscap number to a Hilbert basis enumeration
over a high-dimensional polyhedral cone.
Although this general approach follows a common template in
computational topology, there are complications that cause the usual
theoretical techniques to fail.
To address this, we introduce a special class of triangulations
called \emph{suitable triangulations}, described in
Section~\ref{s-suitable}, with which we are able to
solve these theoretical problems.

The second and third algorithms, described in Sections~\ref{s-exact}
and~\ref{s-inexact}, draw on
techniques from discrete optimisation theory.  Here we develop
an integer programming framework that allows us to approach
problems in computational topology using off-the-shelf optimisation software.
The difference between the second and third algorithms is that
the second requires expensive exact integer arithmetic; the third
makes concessions that allow us to use fast off-the-shelf
solvers based on floating point
computation, but with the side-effect that it
only outputs an upper bound.

The first two algorithms remain intractable for all but the simplest
knots, though each has different strengths through which it may become
practical with the growth of supporting software in algebra and optimisation.
In contrast the third algorithm is extremely fast,
and in Section~\ref{s-comp} we run it over the full $12$-crossing knot tables.
The results are extremely pleasing: of the 2688 knots with
unknown crossing number, for 747 we can improve the best-known
bounds, and for a further 27 we can combine our output with known
lower bounds to compute the crosscap number precisely.

The integer programming framework that we introduce here is
general, and has significant potential for use elsewhere in computational
geometry and topology.  We discuss these matters
further in Section~\ref{s-comp}.

%
%

\section{Preliminaries}

Here we give a short summary of concepts from
knot theory and normal surface theory that appear within this paper.
This outline is necessarily brief; for details see
the excellent overview in \cite{hass99-knotnp}.

For our purposes, a knot is a piecewise linear closed
curve embedded in $\R^3$.
We also treat knots as being embedded in the 3-sphere $S^3$, where
$S^3$ is the one-point compactification $\R^3 \cup \{\infty\}$.
A knot $K$ is typically given as a \emph{knot diagram}, which is a
general position projection of the knot onto the plane as illustrated in
Figure~\ref{fig-trefoil}.  See \cite{hass99-knotnp} for more
precise definitions of these concepts.

\begin{figure}[htb]
    \centering
    \subfigure[]{%
        \label{fig-trefoil}
        \includegraphics[scale=0.9]{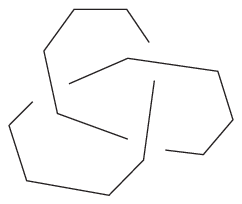}}
    \qquad\qquad
    \subfigure[]{%
        \label{fig-trefoilspan}
        \includegraphics[scale=0.9]{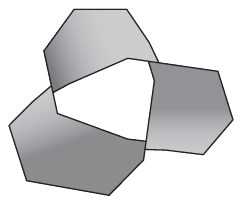}}
    \caption{A knot diagram and a spanning surface for the trefoil knot}
\end{figure}

We are interested in properties of knots related to two-dimensional
surfaces.  In this paper all surfaces are piecewise linear, and might be
disconnected unless otherwise stated.
A surface with no boundary curves (such as a sphere or a torus)
is called \emph{closed}, and a surface with boundary curves (such as a
disc or a {\mobius} band) is called \emph{bounded}.

Every closed orientable surface is topologically a 2-sphere
with $g \geq 0$ orientable
``handles'' attached (i.e., a $g$-holed torus); the \emph{orientable genus}
of such a surface is $g$.
Every closed non-orientable surface is a 2-sphere with $g \geq 1$
``crosscaps'' attached; the \emph{non-orientable genus} of such
a surface is $g$.  We simply use the word \emph{genus}
when orientability is clear from context.
If an (orientable or non-orientable) surface $S$ has boundary, then
the (orientable or non-orientable) genus of $S$ is the genus of the
closed surface obtained by filling each boundary curve of $S$ with a disc.

Let $\mathcal{P}$ be a polygonal decomposition of a surface $S$.
The \emph{Euler characteristic} of $\mathcal{P}$
is defined as $V-E-F$, where $V$, $E$ and $F$ represent
the number of vertices, edges and faces of $\mathcal{P}$ respectively.
Euler characteristic is a topological invariant of a surface, and is
denoted by $\chi(S)$.  The Euler characteristic of an
orientable surface of genus $k$ with $b$ boundary curves is
$2-2k-b$, and the Euler characteristic of a non-orientable surface of
genus $k$ with $b$ boundary curves is $2-k-b$.

The knot invariants that we study here relate to surfaces embedded in
3-dimensional space, which leads us to study 3-manifolds.
A \emph{3-manifold} is a higher-dimensional analogue of a surface: each
interior point of a 3-manifold $M$ has a local neighbourhood that is
topologically similar to $\R^3$, and each point on the boundary of
$M$ has a local neighbourhood that is topologically similar
to the closed half-space $\R^3_{z\geq0}$.
Again, see \cite{hass99-knotnp} for precise details.
A surface $S \subset M$ is \emph{embedded} in $M$ if it has no
self-intersections, and \emph{properly embedded} if in addition
the boundary of $S$ lies within the boundary of $M$, and the
interior of $S$ lies within the interior of $M$.

Consider a knot $K$ embedded in $S^3$.  A \emph{spanning surface} for
$K$ is a connected embedded surface in $S^3$ whose boundary is
precisely $K$.
Figure~\ref{fig-trefoilspan} illustrates a spanning surface for the
trefoil knot (this surface is a {\mobius} band, with non-orientable genus~1).

We can now define the following two invariants of a knot $K$.
The \emph{genus} of $K$, denoted $\genus{K}$, is the smallest $k$ for
which there exists an orientable genus $k$ spanning surface for $K$.
Likewise, the \emph{crosscap number} of $K$, denoted $\cc{K}$, is the
smallest $k$ for which there exists a non-orientable genus $k$ spanning
surface for $K$.  As a special case, the crosscap number of the
trivial knot (also called the \emph{unknot}) is defined to be 0.
The unknot is the only knot with genus 0, and the only knot
with crosscap number 0.  A key relation between genus and
crosscap number is
the following \cite{clark78-crosscaps}:

\begin{theorem}[Clark's inequality] \label{t-clark}
    For any knot $K$, it is true that $\cc{K} \leq 2 \genus{K} + 1$.
\end{theorem}

Let $K \subset S^3$ be a knot, and let $R$ be a small regular
neighbourhood of $K$ in $S^3$.  The \emph{complement} of $K$, denoted
$\comp{K}$, is the closure of $S^3\backslash R$.  This is a 3-manifold with
boundary, obtained by ``eating away'' the knot from $S^3$, as
illustrated in Figure~\ref{fig-knotcomp}.
The boundary surface of $\comp{K}$ is a torus, and any curve on this
torus that bounds a disc in $R$ is called a \emph{meridian}.

\begin{figure}[htb]
    \centering
    \subfigure[]{%
        \label{fig-knotcomp}
        \includegraphics[scale=0.9]{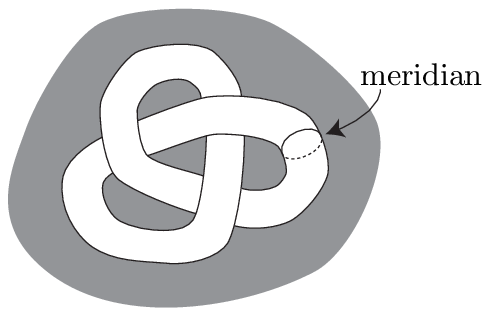}}
    \hspace{2cm}
    \subfigure[]{%
        \label{fig-compspan}
        \includegraphics[scale=0.9]{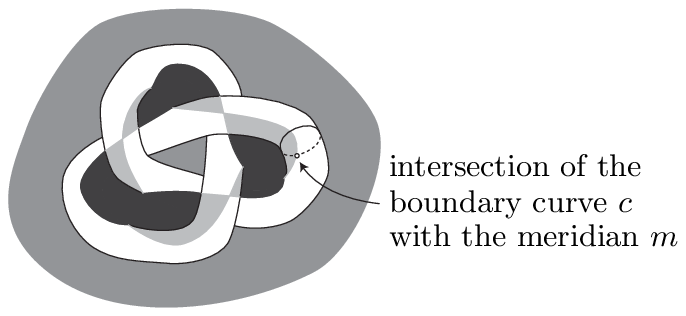}}
    \caption{The complement of the trefoil knot
        and a spanning surface within it}
\end{figure}

We can reformulate spanning surfaces in terms of knot complements.
Consider a knot $K \subset S^3$ and some meridian $m$ on the boundary of
$\comp{K}$.
A \emph{spanning surface} in $\comp{K}$ is a connected,
properly embedded surface $S \subset \comp{K}$ with precisely one
boundary curve $c$, where $c$ and $m$ have algebraic intersection
number $\pm1$ on the torus boundary of $\comp{K}$.\footnote{%
    Essentially we require that the boundary of $S$ cuts $m$ precisely
    once.  Using algebraic intersection number allows us to account
    for any extra trivial ``wiggles'' back and forth across $m$.}
See Figure~\ref{fig-compspan} for an illustration.
This is essentially the same as our previous definition, since
any such surface can be extended through
$R$ to give a connected embedded surface bounded by $K$
and vice versa, assuming that the neighbourhood $R$ is sufficiently small.

All of our algorithms work with triangulations of the complement
$\comp{K}$.  In this paper, a \emph{triangulation} of a 3-manifold is a
collection of $n$ tetrahedra, some of whose $4n$ faces are affinely
identified in pairs.  This broad definition allows for smaller
triangulations than a traditional simplicial complex; in particular,
most triangulations in this paper are \emph{one-vertex triangulations},
where all $4n$ tetrahedron vertices are (as a result of the face
gluings) identified to a single point in $\comp{K}$.

More generally, if $\tri$ is a 3-manifold triangulation, all tetrahedron
vertices that are identified to a single point in $\tri$
are collectively referred to as a
single \emph{vertex} of $\tri$; similarly for edges and
faces.  Any vertex, edge or face that lies in the boundary of the
underlying 3-manifold is called a \emph{boundary} vertex, edge or face;
all others are referred to as \emph{internal}.  Note that the boundary
faces of $\tri$ are precisely those tetrahedron faces that are
not paired with some partner face.

In our algorithms we describe spanning surfaces using normal surface theory.
A \emph{normal surface} in $\tri$ is a properly embedded surface that
meets each tetrahedron $\Delta$ of $\tri$ in a disjoint collection of
triangles and quadrilaterals, each running between
distinct edges of $\Delta$, as illustrated in Figure~\ref{fig-normal}.
There are four \emph{triangle types} and three \emph{quadrilateral types}
according to which edges they meet.
Within each tetrahedron there may be several triangles or quadrilaterals
of any given type; collectively these are referred to as \emph{normal discs}.

\begin{figure}[htb]
    \centering
    \includegraphics[scale=0.5]{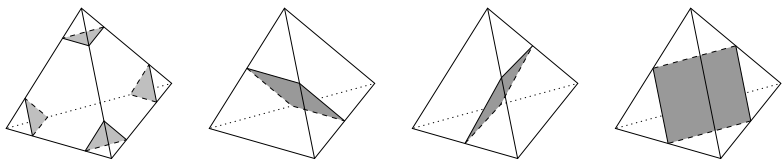}
    \caption{Normal triangles and quadrilaterals within a tetrahedron}
    \label{fig-normal}
\end{figure}

The \emph{vector representation} of a normal surface $S$ is the
$7n$-dimensional integer vector
\[ \vrep{S} = (
        t_{1,1},t_{1,2},t_{1,3},t_{1,4},\ q_{1,1},q_{1,2},q_{1,3}\ ;
        \ t_{2,1},t_{2,2},t_{2,3},t_{2,4},\ q_{2,1},q_{2,2},q_{2,3}\ ;
        \ \ldots,q_{n,3}\ ) \in \Z^{7n}, \]
where $t_{i,j}$ is the number of triangles in the $i$th tetrahedron of
the $j$th type, and $q_{i,k}$ is the number of quadrilaterals in the
$i$th tetrahedron of the $k$th type ($1 \leq i \leq n$, $1 \leq j \leq 4$,
$1 \leq k \leq 3$).

\begin{theorem}[Haken, 1961] \label{t-admissible}
    A vector $\mathbf{v} = (t_{1,1},\ldots,q_{n,3}) \in \R^{7n}$
    is the vector representation of a normal surface if and only if:
    (i)~all elements of $\mathbf{v}$ are non-negative integers;
    (ii)~$\mathbf{v}$ satisfies a certain set of homogeneous linear
    equations derived from the triangulation,
    and
    (iii)~for each $i$, at most one of the three quadrilateral
    coordinates $q_{i,1},q_{i,2},q_{i,3}$ is non-zero.
\end{theorem}

The homogeneous linear equations in (ii) are called the \emph{matching
equations}, and condition (iii) is called the \emph{quadrilateral
constraints}.  The set of all points in $\R^{7n}$ with non-negative
coordinates that satisfy the matching equations is called the
\emph{normal surface solution cone}.

Given a vector that satisfies all of the constraints of
Theorem~\ref{t-admissible},
the corresponding normal surface can be reconstructed uniquely (up to
normal isotopy).  If $X$ and $Y$ are both normal surfaces in some
triangulation, the \emph{normal sum} $X+Y$ is the normal surface with
vector representation $\vrep{X}+\vrep{Y}$.
It is possible that $\vrep{X}+\vrep{Y}$ does not satisfy the
quadrilateral constraints, in which case (for our purposes) the sum
$X+Y$ is not defined.

A \emph{fundamental normal surface} $S$ is one that cannot be written as
$S=X+Y$ for non-empty normal surfaces $X,Y$.
There are finitely many fundamental normal surfaces in a triangulation,
corresponding precisely to the vectors in the
Hilbert basis of the normal surface solution cone.

%
%

\section{Suitable triangulations} \label{s-suitable}

To overcome theoretical difficulties that arise with non-orientable
spanning surfaces, we introduce a special class of
triangulations for our algorithms to use.

\begin{defn}
    A \emph{suitable triangulation} $\tri$
    of a knot complement is one for which:
    \begin{enumerate}[(i)]
        \item $\tri$ has precisely one vertex
            (and therefore the torus boundary of $\tri$ contains
            this one vertex, three edges and two faces);\footnote{%
                This follows by a standard Euler characteristic argument.}
            \label{en-suitable-vtx}
        \item one of the boundary edges of $\tri$ is a meridian.
            \label{en-suitable-meridian}
    \end{enumerate}
\end{defn}

Given a triangulation $\tri$ of a knot complement,
it is easy to test whether $\tri$ is suitable.
Condition~(\ref{en-suitable-vtx})
can be verified by grouping vertices of tetrahedra
into equivalence classes under identification.
For condition~(\ref{en-suitable-meridian}), we can verify that a
boundary edge $e$ is a meridian by attaching a solid torus in an
appropriate fashion
and testing whether the resulting closed manifold is a 3-sphere.\footnote{%
    The solid torus must be attached so that its meridional disc is
    bounded by $e$; this can be done using a $(1,1,0)$ layered solid
    torus as described in \cite{jaco06-layered}.
    Although the subsequent 3-sphere test requires worst-case exponential
    time, it is found to run surprisingly fast
    in practice when the right algorithms
    and simplification heuristics
    are used \cite{burton10-quadoct,burton11-pachner}.}

Our first two algorithms for computing crosscap number
require a further condition:
\begin{defn}
    An \emph{efficient suitable triangulation} $\tri$ of a knot complement
    is a suitable triangulation that contains no
    embedded normal 2-spheres.
\end{defn}

This extra 2-sphere condition is related to, though weaker than,
the 0-efficiency criterion of Jaco and Rubinstein \cite{jaco03-0-efficiency}.
Testing for efficient suitability is possible but slow: to verify the
absence of embedded normal 2-spheres one must
typically enumerate all extreme rays of the high-dimensional
normal surface solution cone \cite{jaco03-0-efficiency}.

We give two algorithms for producing a suitable triangulation of a knot
complement.  The first is general but slow; the second is heuristic in
nature but works extremely well in practice.

\begin{algorithm} \label{alg-suitable}
    Given a knot diagram describing the knot $K \subseteq S^3$,
    the following procedure will output an
    efficient suitable triangulation of $\comp{K}$.
    \begin{enumerate}
        \item Test whether $K$ is the unknot
        \cite{haken61-knot,jaco95-algorithms-decomposition}.
        If so, output a pre-constructed efficient
        suitable triangulation of the unknot complement
        and terminate immediately.
        \label{en-algsuitable-unknot}

        \item Run the procedure of Hass, Lagarias and Pippenger
        \cite[Lemma~7.2]{hass99-knotnp}
        to obtain a triangulation $\tri$ of $\comp{K}$, along with
        a meridian expressed as a path that follows boundary edges
        of $\tri$.

        \item Run the procedure of Jaco and Rubinstein
        \cite[Proposition~5.15 and Theorem~5.20]{jaco03-0-efficiency}
        to convert this into a one-vertex triangulation
        with no embedded normal 2-spheres.
        Keep track of the location of the meridian as the Jaco-Rubinstein
        procedure runs.
        \label{en-algsuitable-0eff}

        \item Apply layerings of extra tetrahedra to alter the
        boundary edges until the meridian consists
        of a single boundary edge, and output the resulting triangulation.

        Each layering attaches a
        new tetrahedron along two boundary faces as illustrated in
        Figure~\ref{fig-layering}.  The number and locations of these layerings
        are determined by a continued fraction calculation, as described
        by Jaco and Rubinstein \cite{jaco06-layered,jaco03-decision}.
        \label{en-algsuitable-layer}
    \end{enumerate}
\end{algorithm}

\begin{figure}[htb]
    \centering
    \includegraphics[scale=0.9]{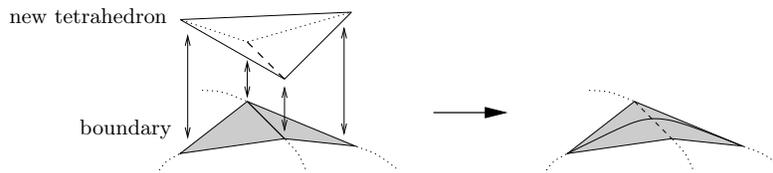}
    \caption{Layering a new tetrahedron onto the boundary}
    \label{fig-layering}
\end{figure}

The unknot test in step~\ref{en-algsuitable-unknot} is necessary
because the Jaco-Rubinstein procedure in step~\ref{en-algsuitable-0eff}
requires $\comp{K}$ to be both irreducible and boundary irreducible,
which is only true for non-trivial knots \cite{hemion92}.
If $K$ is the unknot, the one-tetrahedron solid torus
\cite{jaco03-0-efficiency} can be used as a pre-constructed solution.

\begin{theorem}
    Algorithm~\ref{alg-suitable} is correct, i.e., it produces an
    efficient suitable triangulation of $\comp{K}$.
\end{theorem}

\begin{proof}
    For the correctness of each sub-procedure we refer the reader to the
    source papers
    \cite{haken61-knot,hass99-knotnp,jaco03-0-efficiency,
    jaco06-layered,jaco03-decision,jaco95-algorithms-decomposition}.
    Here we simply prove that the requirements for an efficient suitable
    triangulation are satisfied.

    The Jaco-Rubinstein procedure in step~\ref{en-algsuitable-0eff}
    gives a triangulation with no embedded normal 2-spheres and just one vertex.
    The only missing requirement is that some boundary edge is a
    meridian.  The layering process in step~\ref{en-algsuitable-layer}
    fixes this, and does not break the other requirements:
    \begin{itemize}
        \item Each layering preserves the number of vertices of the
        triangulation.

        \item The layering process does not introduce any new embedded
        normal 2-spheres.  Each time we layer a new tetrahedron onto the
        boundary, every normal triangle or quadrilateral in this new
        tetrahedron meets the boundary of the new triangulation.
        Therefore any embedded normal 2-sphere in the new triangulation cannot
        use these new normal discs, and must
        have been an embedded normal 2-sphere in the old triangulation also.
        \qedhere
    \end{itemize}
\end{proof}

Steps~\ref{en-algsuitable-unknot} and~\ref{en-algsuitable-0eff}
of the previous algorithm are slow; although
step~\ref{en-algsuitable-unknot} can be avoided (e.g., by using prior
information that the input knot is non-trivial), the Jaco-Rubinstein
procedure of step~\ref{en-algsuitable-0eff} remains too inefficient to use
with all but the simplest knot complements.

We therefore offer an alternative, heuristic algorithm that uses only
fast (small polynomial time) operations.
The drawback is that this heuristic algorithm might not produce
any solution at all; however, experience shows this to be a rare
occurrence, as discussed below.

\begin{algorithm} \label{alg-suitable-simp}
    Given a knot diagram describing the knot $K \subseteq S^3$,
    the following heuristic procedure will either output a suitable
    triangulation of $\comp{K}$ or terminate with no output at all.
    \begin{enumerate}
        \item Run the procedure of Hass, Lagarias and Pippenger
        \cite[Lemma~7.2]{hass99-knotnp}
        to obtain a triangulation $\tri$ of $\comp{K}$, along with
        a meridian expressed as a path that follows boundary edges
        of $\tri$.
        \item Simplify the triangulation using fast local operations
        (such as edge collapses, Pachner moves, book closing moves
        and related operations \cite{burton04-regina,burton12-ws})
        to reduce the number of tetrahedra and the number of boundary faces
        as far as possible.
        \label{en-algsuitable-heur-simp}
        \item Test whether the resulting triangulation $\tri'$ is
        suitable.  If so then output $\tri'$, and otherwise terminate
        with no output.
        \label{en-algsuitable-heur-test}
    \end{enumerate}
\end{algorithm}

The choice of local simplification operations in
step~\ref{en-algsuitable-heur-simp} is not important;
see \cite{burton04-regina,burton12-ws} for details.
Although the suitability test in step~\ref{en-algsuitable-heur-test}
requires 3-sphere recognition (which runs in worst-case exponential time),
we can sidestep this by tracking the location of the meridian throughout
step~\ref{en-algsuitable-heur-simp}, avoiding the need to verify
the meridian condition (and test for 3-spheres)
in step~\ref{en-algsuitable-heur-test}.

In Section~\ref{s-comp} we observe that this heuristic
algorithm outputs a suitable triangulation for all 2977
non-trivial prime knots with $\leq 12$ crossings, showing it to be
extremely effective in practice.

We finish this section with some useful properties of suitable
triangulations.

\begin{lemma} \label{l-eff-euler}
    Let $\tri$ be an efficient suitable triangulation of a knot
    complement.  Then every closed normal surface embedded in $\tri$
    has Euler characteristic $\leq 0$.
\end{lemma}

\begin{proof}
    The efficiency criterion ensures that there are no embedded normal
    2-spheres.  The only other closed surface of positive Euler
    characteristic is the projective plane, which does not embed in
    $\R^3$ and so cannot embed in any knot complement.
\end{proof}

\begin{lemma} \label{l-spanning-cut}
    Let $K$ be any knot, let $\tri$ be any suitable triangulation
    of its complement with meridional boundary edge $m$,
    and let $S$ be any normal surface in $\tri$.
    Then $S$ is a spanning surface for $K$ if and only if $S$
    has no closed components and $S$ meets edge $m$ in precisely one point.
\end{lemma}

\begin{proof}
    Let $\bdry \tri$ denote the two-triangle torus on the boundary of
    $\tri$, and let $\bdry S$ denote the collection of boundary curves
    of $S$.  Each curve of $\bdry S$ is a \emph{normal curve}
    on $\bdry \tri$; that is, a union of arcs between distinct edges
    of $\bdry \tri$ as illustrated in Figure~\ref{fig-normalcurves}.

    \begin{figure}[htb]
        \centering
        \includegraphics{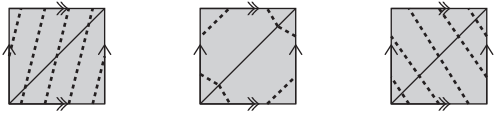}
        \caption{Examples of normal curves on a two-triangle torus}
        \label{fig-normalcurves}
    \end{figure}

    If $S$ is spanning then $S$ is connected and has no
    closed components; moreover, $\bdry S$ consists of a single
    normal curve $c$ whose algebraic intersection number with $m$ is $\pm1$.
    It is a property of normal curves on a two-triangle torus
    that every such curve meets $m$ in precisely one point
    \cite{jaco03-decision}.

    Conversely, suppose that $S$ has no closed components
    and meets $m$ in precisely one point.
    Then $S$ has some
    boundary curve $c$ whose algebraic intersection number with $m$ is
    $\pm1$.  Because the boundary curves of $S$ are disjoint and
    $\bdry \tri$ is a torus, any other
    boundary curve $c'$ of $S$ must be parallel to $c$ or trivial in
    $\bdry \tri$, both of which would generate additional intersections
    with $m$ (note that the only trivial \emph{normal} curve
    on a two-triangle torus cuts each edge twice).
    Therefore $c$ is the only boundary curve of $S$.

    Since $S$ has no closed components it follows that $S$ is a
    connected surface with boundary $c$, and since $c$ has algebraic
    intersection number $\pm1$ with $m$ it follows that $S$ is spanning.
\end{proof}

%
%

\section{A Hilbert basis algorithm} \label{s-hilbert}

Our first algorithm for computing the crosscap number $\cc{K}$
follows a common pattern for topological algorithms: it is based
on the enumeration of fundamental normal surfaces.

It is worth revisiting the algorithm for the orientable
counterpart of $\cc{K}$,
the knot genus $\genus{K}$, as described by Hass, Lagarias
and Pippenger \cite{hass99-knotnp}.  Their algorithm is based on the
following result:
\begin{theorem}[Hass, Lagarias and Pippenger, 1999] \label{t-or-fund}
    Let $K$ be any knot, and $\tri$ be any triangulation of its
    complement.
    Then there is a fundamental normal orientable spanning surface of
    genus $\genus{K}$.
\end{theorem}
The algorithm for computing $\genus{K}$ is then to enumerate all
fundamental normal surfaces in $\tri$, and to observe the smallest genus
orientable spanning surface that appears.

For computing crosscap number, things are less straightforward: it is
not even known whether there must be a \emph{normal} non-orientable
spanning surface of non-orientable genus $\cc{K}$, let alone a
fundamental normal surface.  The arguments of Hass, Lagarias and
Pippenger use the fact that any minimal genus orientable spanning surface
is essential\footnote{%
    That is, both incompressible and boundary incompressible.
    Further details are not required here.};
however, these arguments do not translate to the non-orientable case
since there are knots for which \emph{every}
minimal genus non-orientable spanning surface is non-essential
\cite{bessho94-incompressible,ichihara02-boundary}.

Our solution is twofold: we work with efficient suitable triangulations,
which allow us to obtain precise results in many cases, and for those
cases that remain we use Clark's inequality to reduce the solution to
one of two possible values.

\begin{lemma} \label{l-nor-normal}
    Let $K$ be any non-trivial knot, and $\tri$ be any suitable triangulation
    of its complement.
    Then either there is a normal non-orientable spanning surface of
    non-orientable genus $\cc{K}$, or else $\cc{K} = 2g(K)+1$.
\end{lemma}

\begin{proof}
    Let $S$ be any non-orientable spanning surface of non-orientable
    genus $\cc{K}$, and let $m$ be the meridional boundary edge of $\tri$.
    Since the boundary of $S$ has algebraic intersection number $\pm1$
    with $m$, we can isotope $S$ so that the boundary of $S$ meets
    $m$ in precisely one point (a simple operation
    on the two-triangle torus boundary of $\tri$).

    We now follow the standard \emph{normalisation procedure} that converts
    an arbitrary properly embedded surface into an embedded normal surface
    (possibly with different topology).
    We do not reiterate the details of this procedure here;
    for full details the reader is referred to a standard reference
    such as \cite{jaco89-surgery} or \cite{matveev03-algms}.
    Instead we highlight some of its key aspects:

    \begin{itemize}
        \item Most steps in the procedure are isotopies,
        which preserve the
        fact that $S$ is a non-orientable spanning surface of genus $\cc{K}$.
        However, three types of step can alter the topology of $S$:
        \begin{itemize}
            \item \emph{internal compressions},
            which involve surgery on a disc that is bounded by a curve on $S$,
            as illustrated in Figure~\ref{fig-icompress};
            \item \emph{boundary compressions},
            which involve surgery on a disc that is
            bounded by an arc on $S$ and an arc on the boundary of $\tri$,
            as illustrated in Figure~\ref{fig-bcompress};
            \item \emph{deletion of trivial components},
            where we remove components of $S$ that are trivial spheres
            (bounding a ball in $\comp{K}$) or trivial discs
            (parallel into the boundary of $\comp{K}$).
        \end{itemize}

        \begin{figure}[htb]
            \centering
            \subfigure[An internal compression]{%
                \label{fig-icompress}
                \includegraphics[scale=0.9]{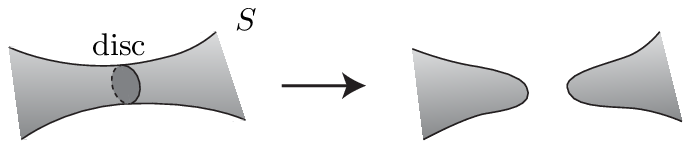}}
            \qquad\qquad
            \subfigure[A boundary compression]{%
                \label{fig-bcompress}
                \includegraphics[scale=0.9]{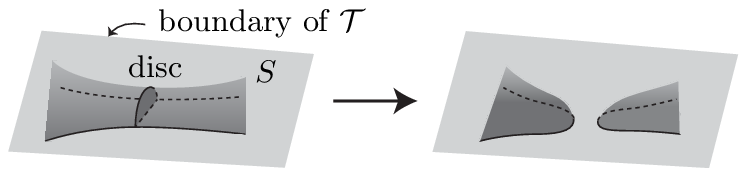}}
            \caption{Examples of normalisation moves
                that can alter the topology of $S$}
        \end{figure}

        \item For any edge $e$ of $\tri$, no step in the procedure
        will ever increase the number of intersections between the surface
        and $e$.  Moreover, each internal compression and
        each boundary compression
        can be performed in a manner that never reduces this number of
        intersections either.
    \end{itemize}

    We now analyse how internal compressions and boundary
    compressions can affect our non-orientable spanning surface $S$.
    At each stage, we assume that $S$ is a non-orientable spanning
    surface of minimum genus $\cc{K}$.
    \begin{itemize}
        \item For internal compressions, there are two cases to
        consider:
        \begin{itemize}
            \item If the boundary of the compression disc is a
            separating curve in $S$, then the compression splits $S$ into two
            disjoint pieces $S_1$ and $S_2$, where $S$ is the connected
            sum $S_1 \# S_2$.  Without loss of generality,
            we assume that the boundary
            curve of $S$ stays with $S_1$, and so $S_1$ is also a
            spanning surface and $S_2$ is a closed surface.

            Since $S=S_1 \# S_2$, at least one of $S_1$ and $S_2$ must
            be non-orientable, and since no closed non-orientable
            surface can embed in a knot complement, this
            non-orientable piece must be $S_1$.
            Moreover, a simple Euler characteristic calculation gives
            $\chi(S)=\chi(S_1)+\chi(S_2)-2$, and since $\chi(S_2)\leq2$
            it follows that $S_1$ (like $S$) must have the smallest possible
            non-orientable genus $\cc{K}$.
            Therefore we can simply delete $S_2$, replace $S$ with the
            non-orientable spanning surface $S_1$
            of genus $\cc{K}$, and continue the process.

            \item If the boundary of the compression disc is not
            separating in $S$, then the resulting surface is some
            spanning surface $U$ for which $\chi(U)=\chi(S)+2$.
            $U$ cannot be non-orientable, since its non-orientable genus
            would be less than the minimum $\cc{K}$.  Therefore $U$ is
            an orientable spanning surface with orientable genus
            $\frac12(1-\chi(U))=\frac12(1-\chi(S)-2)=\frac12(\cc{K}-2)$.
            This gives $\genus{K} \leq \frac12(\cc{K}-2)$ in contradiction to
            Clark's inequality (Theorem~\ref{t-clark}), and so this case
            can never occur.
        \end{itemize}

        \item For boundary compressions we first note that,
        since boundary compressions do not change the number of
        intersections with the meridional edge $m$,
        the resulting (possibly disconnected) surface still cuts $m$
        precisely once.
        \begin{itemize}
            \item If the compression separates $S$ into two disjoint
            surfaces $S_1$ and $S_2$, we may assume that $S_1$ cuts edge
            $m$ precisely once and $S_2$ does not cut $m$ at all.
            Therefore the boundary of $S_2$ must be either a meridian or
            a trivial curve on the boundary torus; either way, it can
            be filled with a disc in $S^3 \backslash \comp{K}$ to yield a
            closed surface in $S^3$.  Since no closed non-orientable
            surface can embed in $S^3$, it follows that $S_2$ must be
            orientable.

            Therefore the piece $S_1$ is non-orientable.  This time the
            Euler characteristic argument gives
            $\chi(S)=\chi(S_1)+\chi(S_2)-1$ with $\chi(S_2)\leq 1$.
            Therefore the piece $S_1$ is again a non-orientable spanning
            surface with genus $\cc{K}$, and we simply delete $S_2$,
            replace $S$ with $S_1$ and continue.

            \item If the compression does not separate $S$, the
            result is a spanning surface $U$ with $\chi(U)=\chi(S)+1$.
            In this case the normalisation process might fail.
            As before, $U$ cannot have non-orientable genus less than
            $\cc{K}$; therefore $U$ is an \emph{orientable} spanning surface
            with orientable genus $\frac12(1-\chi(U))=\frac12(\cc{K}-1)$.
            This gives $\genus{K} \leq \frac12(\cc{K}-1)$, and combined with
            Clark's inequality we obtain $\cc{K}=2\genus{K}+1$ precisely.
        \end{itemize}
    \end{itemize}

    If we follow the arguments above, deletion of trivial components is
    never required: we delete unwanted extra components as they appear
    during internal and boundary compressions, and the spanning surface
    itself is never a trivial sphere nor a trivial disc.

    In conclusion, either the normalisation procedure yields
    an embedded \emph{normal} non-orientable
    spanning surface of genus $\cc{K}$, or else we have a situation in which
    $\cc{K}=2\genus{K}+1$.
\end{proof}

We can now move from normal surfaces to \emph{fundamental} normal surfaces,
giving us the main theorem of this section.

\begin{theorem} \label{t-nor-fund}
    Let $K$ be any non-trivial knot,
    and $\tri$ be any efficient suitable triangulation of its complement.
    Then either there is a \emph{fundamental}
    normal non-orientable spanning surface of
    non-orientable genus $\cc{K}$, or else
    $\cc{K} \in \{2\genus{K},\ 2\genus{K}+1\}$.
\end{theorem}

\begin{proof}
    If there is no \emph{normal} non-orientable spanning surface of
    genus $\cc{K}$ then $\cc{K} = 2\genus{K}+1$, by
    Lemma~\ref{l-nor-normal}.
    Assume then that there is some normal non-orientable spanning
    surface of genus $\cc{K}$, and let $S$ be such a surface containing
    the fewest possible normal discs.  Let $m$ denote the meridional
    boundary edge of $\tri$.

    If $S$ is not fundamental then $S=U+V$ for some non-empty
    normal surfaces $U$ and $V$.  By Lemma~\ref{l-spanning-cut}, $S$ cuts the
    edge $m$ in precisely one point.  Since edge intersections
    are additive under normal sum\footnote{%
        This refers to absolute numbers of intersections,
        not algebraic intersection number.},
    we may assume without loss of generality that $U$ cuts $m$ precisely
    once and $V$ does not cut $m$ at all.  We may also assume that $U$
    is connected (since any components that do not meet $m$ can be moved
    over to $V$), whereby Lemma~\ref{l-spanning-cut}
    shows that $U$ is a spanning surface also.

    The surface $V$ is a union of closed components and/or bounded components.
    Every boundary curve of $V$ is disjoint from $m$ and is therefore a
    meridian (since any other \emph{normal} boundary curve, even the trivial
    curve, must cut $m$ at least once).  This means that no bounded component
    of $V$ can be a disc (because the meridian has non-trivial homology in
    $\comp{K}$, and so no disc in $\comp{K}$ can have a meridian as boundary).
    Therefore every bounded component of $V$ has non-positive
    Euler characteristic; combined with Lemma~\ref{l-eff-euler} this
    gives us $\chi(V) \leq 0$.
    Since Euler characteristic is additive under normal sum we have
    $\chi(S)=\chi(U)+\chi(V)$, and therefore $\chi(U) \geq \chi(S)$.

    If $U$ is non-orientable, it follows that $U$ is a normal non-orientable
    spanning surface of genus at most $\cc{K}$.  By minimality of $\cc{K}$
    this genus is precisely $\cc{K}$; moreover, since $S=U+V$ we see that $U$
    contains fewer normal discs than $S$.  This contradicts our initial
    choice of $S$.
    
    Therefore $U$ is orientable, and has genus
    $\leq \frac12(1-\chi(S)) = \frac12 \cc{K}$.
    Since $U$ is spanning this gives $2 \genus{K} \leq \cc{K}$,
    and by Clark's inequality it follows that
    $\cc{K} \in \{2\genus{K},\ 2\genus{K}+1\}$.
\end{proof}

Combining Theorems~\ref{t-or-fund} and \ref{t-nor-fund}
gives our first algorithm for computing crosscap number.

\begin{algorithm} \label{a-cc-normal}
    Given a knot diagram describing the knot $K \subseteq S^3$,
    the following procedure will either (i)~output the crosscap
    number $\cc{K}$, or (ii)~output a pair of integers one of which is
    $\cc{K}$.
    \begin{enumerate}
        \item Construct an efficient suitable triangulation $\tri$ of the
        complement $\comp{K}$ using Algorithm~\ref{alg-suitable}.

        \item Enumerate the set $\mathcal{F}$ of all fundamental
        normal surfaces in $\tri$.
        \label{en-alg-normal-fund}

        \item For each surface $S \in \mathcal{F}$, test whether $S$ is a
        spanning surface for $K$ and whether $S$ is orientable or
        non-orientable.
        Let $g_n$ be the minimum non-orientable genus amongst all
        non-orientable spanning surfaces in $\mathcal{F}$, or $\infty$
        if no non-orientable spanning surfaces are found.
        Let $g_o$ be the minimum orientable genus amongst all orientable
        spanning surfaces in $\mathcal{F}$.
        \label{en-alg-normal-test}

        \item If $g_o=0$ then output $0$.
        If $g_n \leq 2g_o$ then output $g_n$.
        Otherwise output $\{2g_o,\ 2g_o+1\}$.
        \label{en-alg-normal-out}
    \end{enumerate}
\end{algorithm}

As outlined in \cite{hass99-knotnp}, the enumeration in
step~\ref{en-alg-normal-fund} involves a high-dimensional Hilbert
basis computation, and can be performed using standard software such as
{\normaliz} \cite{bruns10-normaliz}.
In step~\ref{en-alg-normal-test}, we can test for orientability and
genus by reconstructing the normal surface, and we can test
whether $S$ is a spanning surface for $K$ using
Lemma~\ref{l-spanning-cut}.

\begin{theorem}
    Algorithm~\ref{a-cc-normal} is correct, i.e.,
    the crosscap number $\cc{K}$ is
    one of the solutions that it outputs in step~\ref{en-alg-normal-out}.
\end{theorem}

\begin{proof}
    By Theorem~\ref{t-or-fund}, $g_o = \genus{K}$.
    If $K$ is the unknot then $g_o=0$ and the algorithm
    outputs the correct value
    $\cc{K}=0$.  Assume then that $K$ is non-trivial, which means
    that $g_o>0$.

    If some non-orientable spanning surface of genus
    $\cc{K}$ does appear in the set $\mathcal{F}$ then we also have
    $g_n = \cc{K}$.  If $g_n \leq 2g_o$ then we output the correct value
    $g_n=\cc{K}$.  Otherwise Clark's inequality (Theorem~\ref{t-clark})
    gives $\cc{K}=2g_o+1$, and so the correct value of $\cc{K}$ appears in
    the output set $\{2g_o,\ 2g_o+1\}$.

    If no non-orientable spanning surface of genus $\cc{K}$
    appears in $\mathcal{F}$ then $g_n > \cc{K}$, and by
    Theorem~\ref{t-nor-fund} we also have
    $\cc{K} \in \{2\genus{K},\ 2\genus{K}+1\} = \{2g_o,\ 2g_o+1\}$.
    In this case $2g_o < g_n$, and again our output set
    $\{2g_o,\ 2g_o+1\}$ contains the correct value of $\cc{K}$.
\end{proof}

%
%

\section{An exact integer programming algorithm} \label{s-exact}

The bottleneck in Algorithm~\ref{a-cc-normal}
is the Hilbert basis computation (the enumeration of all
fundamental normal surfaces), which remains intractable for all
but the simplest knots.
For this reason we seek alternate algorithms that do not require an
enumeration of surfaces.  Instead we attempt to locate a minimal genus
non-orientable spanning surface directly using integer programming.

We begin this section with several results that allow us to
formulate topological constraints numerically.
We finish with Algorithm~\ref{a-cc-ipexact}, which gives the
full procedure for computing $\cc{K}$.

The first lemma is due to Hass, Lagarias and Pippenger \cite{hass99-knotnp},
and allows us to place upper bounds on the coordinates of fundamental
normal surfaces.

\begin{lemma}[Hass, Lagarias and Pippenger, 1999] \label{l-ubound}
    Let $\tri$ be a 3-manifold triangulation with $n$ tetrahedra, and
    let $S$ be a fundamental normal surface in $\tri$.
    Then every coordinate of the vector representation $\vrep{S}$
    is at most $n \cdot 2^{7n+2}$.
\end{lemma}

This bound plays an important role in our algorithm,
and it is desirable to improve it.
Techniques based on linear programming offer potential; for instance,
in \cite{burton10-tree} the authors significantly reduce this
bound for related problems.\footnote{%
    The paper \cite{burton10-tree} considers \emph{vertex} (not
    fundamental) normal surfaces, and works in a different coordinate
    system.}
We do not pursue these matters further here.

Our next result is the well-known observation that Euler characteristic
is linear on the normal surface solution cone.  Our proof
includes details on how such a linear function can be constructed.

\begin{lemma} \label{l-ec-linear}
    Let $\tri$ be an $n$-tetrahedron 3-manifold triangulation.
    Then there is a linear function $\chi\co\R^{7n} \to \R$ such that,
    for any normal surface $S$ in $\tri$, $\chi(\vrep{S})$ is the Euler
    characteristic of $S$.
\end{lemma}

\begin{proof}
    Let $\mathbf{x} \in \R^{7n}$.  We define
    $\chi(\mathbf{x})=V-E-F$, where:
    \begin{itemize}
        \item $F$ is the sum of all $4n$ triangular coordinates
        and all $3n$ quadrilateral coordinates of $\mathbf{x}$.
        \item $E$ is defined as a sum over edges of normal discs:
        if $e$ is one of the three edges of a normal triangle
        or one of the four edges of a normal quadrilateral,
        then $e$ contributes $+1$ to this sum if $e$ lies on the
        boundary of $\tri$, or $+1/2$ if $e$ is internal to $\tri$.
        \item $V$ is defined as a sum over vertices of normal discs:
        if $v$ is one of the three vertices of a normal triangle
        or one of the four vertices of a normal quadrilateral,
        then $v$ contributes $+1/d$ to this sum, where $d$ is the degree
        of the \emph{edge} of $\tri$ that $v$ lies within.
    \end{itemize}
    It is clear that $V$, $E$ and $F$ can all be expressed as sums of
    normal coordinates, and so $\chi$ is linear on $\R^{7n}$.
    Moreover: each vertex of degree $d$ in a normal surface $S$
    contributes a total of $+1/d \times d = +1$ to $V$ (since it meets
    $d$ normal discs, and it must lie within an edge of $\tri$ of degree $d$);
    each boundary edge of $S$ contributes $+1$ to $E$;
    each internal edge of $S$ contributes $+1/2 \times 2 = +1$ to $E$
    (since it meets two normal discs);
    and each triangular or quadrilateral face of $S$ contributes
    $+1$ to $F$.  Therefore $\chi(\vrep{S})=V-E+F$ is indeed the Euler
    characteristic of $S$.
\end{proof}

It should be noted that there are many linear functions
$\chi\co\R^{7n} \to \R$ with this property; the formulation above was
chosen for its simple proof.  There are sparser formulations that may
be preferable for computation; again we do not pursue this matter here.

Our final lemma allows us to arithmetically express the condition that a
given normal surface is a spanning surface for our knot.

\begin{lemma} \label{l-spanning-eqn}
    Let $K$ be any knot, and let $\tri$ be any suitable triangulation
    of its complement with meridional boundary edge $m$.
    Let $\Delta_i$ be any tetrahedron of $\tri$ that contains the
    boundary edge $m$, and let $t_{i,a}$, $t_{i,b}$, $q_{i,c}$ and $q_{i,d}$ be
    the coordinates describing the four normal disc types in $\Delta_i$
    that touch $m$,
    as illustrated in Figure~\ref{fig-discscutedge}.

    For any normal surface $S$ in $\tri$,
    $S$ is a spanning surface for $K$ if and only if $S$ has no closed
    components and the vector representation of $S$ satisfies
    $t_{i,a}+t_{i,b}+q_{i,c}+q_{i,d} = 1$.
    We call this equation the \emph{spanning equation} for $\tri$.
\end{lemma}

\begin{figure}[htb]
    \centering
    \includegraphics[scale=0.9]{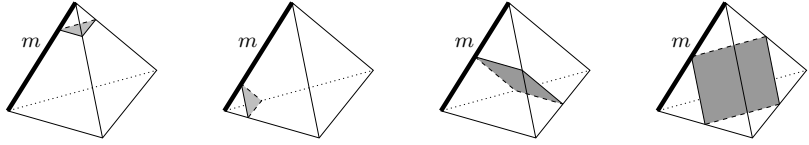}
    \caption{The four normal disc types in $\Delta_i$ that touch
        the meridional boundary edge $m$}
    \label{fig-discscutedge}
\end{figure}

\begin{proof}
    The sum $t_{i,a}+t_{i,b}+q_{i,c}+q_{i,d}$ counts the number of times
    that the surface $S$ cuts the meridional boundary edge $m$ (this is
    true regardless of which tetrahedron $\Delta_i$ we chose).
    The result then follows immediately from Lemma~\ref{l-spanning-cut}.
\end{proof}

\begin{algorithm} \label{a-cc-ipexact}
    Given a knot diagram describing the knot $K \subseteq S^3$,
    the following procedure will either (i)~output the crosscap
    number $\cc{K}$, or (ii)~output a pair of integers one of which is
    $\cc{K}$.
    \begin{enumerate}
        \item Construct an efficient suitable triangulation $\tri$ of the
        complement $\comp{K}$ using Algorithm~\ref{alg-suitable}.
        Let $n$ be the number of tetrahedra in $\tri$.

        \item Optimise the following integer program
        using an exact arithmetic integer programming solver:
        \begin{enumerate}[(i)]
            \item define the $7n$ non-negative integer variables
            $t_{i,j}$ for $1 \leq i \leq n$, $1 \leq j \leq 4$ and
            $q_{i,j}$ for $1 \leq i \leq n$, $1 \leq j \leq 3$;
            \item define the $3n$ binary variables
            $b_{i,j}$ for $1 \leq i \leq n$, $1 \leq j \leq 3$;
            \item add constraints for the matching equations
            and the spanning equation (Lemma~\ref{l-spanning-eqn});
            \item add constraints
            $n \cdot 2^{7n+2} \cdot b_{i,j} \geq q_{i,j}$
            for $1 \leq i \leq n$, $1 \leq j \leq 3$;
            \item add constraints
            $b_{i,1}+b_{i,2}+b_{i,3} \leq 1$ for $1 \leq i \leq n$;
            \item maximise the Euler characteristic function
            (Lemma~\ref{l-ec-linear}).
        \end{enumerate}
        \label{en-algipexact-opt}

        \item Let $\mathbf{x}$ be an
        optimal solution, and let $S$ be the corresponding normal
        surface with any closed components removed.
        If $S$ is non-orientable then output $1-\chi(S)$.
        Otherwise output the pair $\{1-\chi(S), 2-\chi(S)\}$.
        \label{en-algipexact-out}
    \end{enumerate}
\end{algorithm}

It is crucial that the integer programming solver be based on
exact integer arithmetic;
see \cite{applegate07-qsopt,cook11-exact-lncs} for examples of such tools.
With traditional floating-point solvers, round-off errors can creep
in, especially when working with coefficients as large as
$n \cdot 2^{7n+2}$.  Such round-off errors can result in a
sub-optimal solution, or even a solution $\mathbf{x}$ that
does not represent a spanning surface at all.

\begin{theorem}
    Algorithm~\ref{a-cc-ipexact} is well-defined (i.e.,
    the integer program in step~\ref{en-algipexact-opt} is not
    unbounded), and is correct (i.e., the crosscap number $\cc{K}$ is
    one of the solutions output in step~\ref{en-algipexact-out}).
\end{theorem}

\begin{proof}
    The integer variables $t_{i,j}$ and $q_{i,j}$ correspond to the usual
    triangle and quadrilateral coordinates.
    Each binary variable $b_{i,j}$ has the following effect under
    condition (iv):
    if $b_{i,j}=0$ then it forces $q_{i,j}=0$, and if $b_{i,j}=1$ then
    we can have any $q_{i,j}$ in the range
    $0 \leq q_{i,j} \leq n \cdot 2^{7n+2}$.

    First consider any feasible solution to the integer program (optimal or
    otherwise).  Because any such solution satisfies
    all constraints from step~\ref{en-algipexact-opt}, we know that the
    variables $t_{i,j}$ and $q_{i,j}$ are non-negative integers, satisfy
    the matching equations and the spanning equation (condition~(iii)),
    and satisfy the quadrilateral constraints (condition~(v)).
    By Theorem~\ref{t-admissible} and Lemma~\ref{l-spanning-eqn}
    it follows that any feasible solution to our integer program describes
    a normal surface in $\tri$, which is the disjoint union of a
    spanning surface and zero or more additional closed components.

    If the integer program is unbounded then
    such a union can have arbitrarily large Euler characteristic.
    Because any spanning
    surface has Euler characteristic $\leq 1$, it follows that
    $\tri$ contains some closed normal surface of positive
    Euler characteristic.  However, this is impossible by
    Lemma~\ref{l-eff-euler}, and so the integer program is bounded
    and Algorithm~\ref{a-cc-ipexact} is well-defined.

    We now turn to the proof of correctness.
    By the argument above, the surface
    $S$ obtained in step~\ref{en-algipexact-out} (with closed
    components removed) must be a spanning surface.
    If $S$ is non-orientable then its non-orientable genus is
    $1-\chi(S)$, and we have $\cc{K} \leq 1-\chi(S)$.
    If $S$ is orientable then its orientable genus is
    $\frac12(1-\chi(S)) \geq \genus{K}$, and by Clark's inequality we have
    $\cc{K} \leq 2\genus{K}+1 \leq 2-\chi(S)$.

    Having bounded $\cc{K}$ from above, we now bound it from below.
    Let $S_o$ be the smallest-genus orientable
    normal surface whose normal coordinates are
    all at most $n \cdot 2^{7n+2}$,
    and let $S_n$ be the smallest-genus
    non-orientable normal surface whose normal
    coordinates are at most $n \cdot 2^{7n+2}$.
    By Theorem~\ref{t-or-fund} and
    Lemma~\ref{l-ubound}, the surface $S_o$ exists and has genus $\genus{K}$.
    By Theorem~\ref{t-nor-fund} and Lemma~\ref{l-ubound}, we have one
    of two cases for $S_n$: either
    (a)~$S_n$ exists and has non-orientable genus $\cc{K}$, or
    (b)~$\cc{K} \geq 2\genus{K}$ (in which case $S_n$ might or might not exist).

    Since $S_o$ and $S_n$ are normal spanning surfaces with coordinates
    $\leq n \cdot 2^{7n+2}$, they satisfy all of the constraints
    in step~\ref{en-algipexact-opt} of the algorithm
    (here each $b_{i,j}=0$ or $1$ according to whether the
    corresponding $q_{i,j}$ is zero or non-zero).  If $S'$ is the normal
    surface corresponding to the optimal solution $\mathbf{x}$,
    it follows that $\chi(S_o) \leq \chi(S')$,
    and that $\chi(S_n) \leq \chi(S')$ if $S_n$ exists.
    As before, $\tri$ contains no closed normal surfaces
    of positive Euler characteristic, and so after removing any closed
    components of $S'$ we obtain
    $\chi(S_o) \leq \chi(S)$, and $\chi(S_n) \leq \chi(S)$ if $S_n$ exists.

    We can now piece our various results together.
    In case~(a) above where $S_n$ exists with non-orientable
    genus $\cc{K}$, we have $\cc{K} = 1-\chi(S_n) \geq 1-\chi(S)$.
    In case~(b) we have
    $\cc{K} \geq 2\genus{K}=1-\chi(S_o) \geq 1-\chi(S)$.
    Either way we obtain the lower bound $1-\chi(S) \leq \cc{K}$.

    In conclusion: if $S$ is orientable then
    $1-\chi(S) \leq \cc{K} \leq 2-\chi(S)$, and if $S$ is
    non-orientable then
    $1-\chi(S) \leq \cc{K} \leq 1-\chi(S)$.
    Therefore the output of our algorithm is correct.
\end{proof}

%
%

\section{A limited precision integer programming algorithm} \label{s-inexact}

The exact integer programming algorithm in the previous section
avoids an expensive Hilbert basis enumeration, but it has its own
drawbacks.  Exact integer programming solvers are rarer and less
well-developed than their limited precision floating-point cousins,
and the exponentially large constraint coefficients $n \cdot 2^{7n+2}$
can have a crippling effect on performance.
Moreover, constructing an \emph{efficient} suitable triangulation
remains slow, as discussed in Section~\ref{s-suitable}.

Our final algorithm avoids these performance problems: we allow
floating-point solvers with limited precision, and we replace each large
coefficient $n \cdot 2^{7n+2}$ in our integer program
with the arbitrarily chosen small coefficient
$10000$.  We also drop the efficiency requirement and allow just
suitable triangulations, which are faster to construct.

These concessions bring about a loss of information and precision from our
solution.  In response, we add tests to ensure that the solution
to our integer program is valid (i.e., represents a spanning surface),
and we interpret the final output value as just an upper bound on
$\cc{K}$ (since we
cannot be sure that a better solution was inadvertently missed).
The complete algorithm is as follows.

\begin{algorithm} \label{a-cc-ipfloat}
    Given a knot diagram describing the knot $K \subseteq S^3$,
    the following procedure will output an upper bound $U$ for which
    $\cc{K} \leq U$.
    \begin{enumerate}
        \item Construct a suitable triangulation $\tri$ of the
        complement $\comp{K}$ using the fast but heuristic-based
        Algorithm~\ref{alg-suitable-simp}.
        If this algorithm produces no triangulation then
        output $\infty$ and terminate immediately.
        Otherwise let $n$ be the number of tetrahedra in $\tri$.

        \item Optimise the following integer program
        using a fast solver based on floating-point arithmetic:
        \begin{enumerate}[(i)]
            \item define the $7n$ non-negative integer variables
            $t_{i,j}$ for $1 \leq i \leq n$, $1 \leq j \leq 4$ and
            $q_{i,j}$ for $1 \leq i \leq n$, $1 \leq j \leq 3$;
            \item define the $3n$ binary variables
            $b_{i,j}$ for $1 \leq i \leq n$, $1 \leq j \leq 3$;
            \item add constraints for the matching equations
            and the spanning equation (Lemma~\ref{l-spanning-eqn});
            \item add constraints
            $10000 \cdot b_{i,j} \geq q_{i,j}$
            for $1 \leq i \leq n$, $1 \leq j \leq 3$;
            \item add constraints
            $b_{i,1}+b_{i,2}+b_{i,3} \leq 1$ for $1 \leq i \leq n$;
            \item maximise the Euler characteristic function
            (Lemma~\ref{l-ec-linear}).
        \end{enumerate}
        \label{en-algipfloat-opt}

        \item If this integer program is unbounded then
        output $\infty$ and terminate.

        \item Let $\mathbf{x}$ be an optimal solution.
        Using exact integer arithmetic,
        test whether $\mathbf{x}$ satisfies the constraints
        laid out in step~\ref{en-algipfloat-opt}.  If not then output
        $\infty$ and terminate.
        \label{en-algipfloat-test}

        \item Let $S$ be the normal surface described by $\mathbf{x}$
        with any closed components removed.
        Output $1-\chi(S)$ if $S$ is non-orientable,
        or $2-\chi(S)$ if $S$ is orientable.
    \end{enumerate}
\end{algorithm}

\begin{theorem}
    Algorithm~\ref{a-cc-ipfloat} is correct, i.e.,
    $\cc{K}$ is less than or equal to the output value.
\end{theorem}

\begin{proof}
    This is a much simpler variant of the proof for
    Algorithm~\ref{a-cc-ipexact}.  If we output $\infty$ then the
    algorithm is clearly correct.  Otherwise we have a solution
    $\mathbf{x}$ that
    satisfies the constraints of step~\ref{en-algipfloat-opt}
    (as shown by the test in step~\ref{en-algipfloat-test}),
    but with no guarantee that this
    solution is optimal.

    As in Algorithm~\ref{a-cc-ipexact}, because $\mathbf{x}$
    satisfies the constraints from step~\ref{en-algipexact-opt}, the
    surface $S$ (with closed components removed) must be a spanning
    surface for $K$.  If $S$ is non-orientable then its non-orientable
    genus is $1-\chi(S)$, and so $\cc{K} \leq 1-\chi(S)$.
    If $S$ is orientable then its orientable genus is
    $\frac12(1-\chi(S)) \geq \genus{K}$, and by Clark's inequality we have
    $\cc{K} \leq 2\genus{k}+1 \leq 2-\chi(S)$.
\end{proof}

%
%

\section{Discussion and computational results} \label{s-comp}

As discussed in the introduction, Algorithms~\ref{a-cc-normal} and
\ref{a-cc-ipexact} remain too slow for practical use for all but the
simplest knots.  Algorithm~\ref{a-cc-normal} requires the
enumeration of a Hilbert basis for a high-dimensional polyhedral cone,
an extremely expensive procedure.
Algorithm~\ref{a-cc-ipexact} requires exact arithmetic integer
programming in problems with extremely large coefficients
(such as $n \cdot 2^{7n+2}$), which remains intractable with currently
available tools.

Despite this, both algorithms have certain benefits.  Unlike our final
algorithm, they produce no more than two possible solutions;
moreover, further analysis shows that if
$\cc{K} < 2\genus{K}$ then both Algorithms~\ref{a-cc-normal}
and \ref{a-cc-ipexact} guarantee a unique solution.
They also rely on substantially different underlying
computational problems (Hilbert basis enumeration versus exact integer
programming), both of which have supporting software that continues
to enjoy significant advances in efficiency
\cite{applegate07-qsopt,bruns10-normaliz,cook11-exact-lncs};
in this sense we are able to ``hedge our bets''.
A further avenue for improving the efficiency of Algorithm~\ref{a-cc-ipexact}
is to lower the coordinate bounds in Lemma~\ref{l-ubound}.

An important observation is that none of the algorithms in this paper
are able to give a unique solution in cases where Clark's inequality
is tight, i.e., $\cc{K}=2\genus{K}+1$.

In contrast to the first two algorithms, the limited-precision
Algorithm~\ref{a-cc-ipfloat} is fast and effective.  By combining the
upper bounds from this algorithm with lower bounds
from existing knot tables, we are able to make significant improvements
to these tables.

Specifically: the {\knotinfo} project contains a rich body of invariant
data for all 2977 prime non-trivial knots with $\leq 12$ crossings.
Of these knots, 289 have known crosscap numbers, and the remaining
2688 knots are listed with best-known upper and lower bounds.

For each of these knots, we begin with the corresponding
triangulation of the complement from the {\snappy} census \cite{snappy},
use {\regina} \cite{burton04-regina,regina}
to convert it into a triangulation with boundary faces,
and then run Algorithm~\ref{a-cc-ipfloat}
to obtain a bound on the crosscap number.
For the crucial optimisation step we use IBM's {\cplex} package
(version~12.2).
The triangulations range from 5 to 50 tetrahedra in size (with an
average of $33.27$), and the total running time over all triangulations is
roughly $4.5$ hours on a quad-core 2.93\,GHz Intel Core i7 CPU.

Our first observation is that Algorithm~\ref{a-cc-ipfloat}
never outputs $\infty$.  That is: for \emph{every} knot in the
tables, the heuristic Algorithm~\ref{alg-suitable-simp} produces a
suitable triangulation, the integer program is not unbounded (giving
strong evidence that these triangulations may in fact be \emph{efficient}
suitable triangulations), and the optimal solution to the integer
program satisfies all of the necessary constraints to produce a valid
spanning surface.

For 27 of the 2688 knots with unknown crossing number, the upper bound
produced by Algorithm~\ref{a-cc-ipfloat} is equal to the lower bound
listed in the {\knotinfo} tables; as a result we can identify the
crosscap number precisely.  These 27 knots and their crosscap numbers
are listed in Table~\ref{tab-new}.
For another 747 knots, our upper bound
improves upon the {\knotinfo} upper bound: on average we reduce the number
of possible solutions from $4.18$ to $2.86$.
Detailed results from all of these computations,
including a {\regina} data file listing the final spanning surfaces,
can be found at \url{http://www.maths.uq.edu.au/~bab/code/}.

\begin{table}
\centering
{\small
\begin{tabular}{c|c|c|c|c}
KnotInfo & Dowker-Thistlethwaite &
Genus & Previous bounds & New value \\
name & name & & on $C(K)$ & of $C(K)$ \\
\hline
$\mathrm{8}_{20}$ & $\mathrm{8n}_{1}$ & $2$ & $[2,4]$ & $2$ \\
$\mathrm{10}_{125}$ & $\mathrm{10n}_{15}$ & $3$ & $[2,4]$ & $2$ \\
$\mathrm{10}_{126}$ & $\mathrm{10n}_{17}$ & $3$ & $[2,4]$ & $2$ \\
$\mathrm{10}_{139}$ & $\mathrm{10n}_{27}$ & $4$ & $[2,3]$ & $2$ \\
$\mathrm{10}_{140}$ & $\mathrm{10n}_{29}$ & $2$ & $[2,4]$ & $2$ \\
$\mathrm{10}_{142}$ & $\mathrm{10n}_{30}$ & $3$ & $[2,4]$ & $2$ \\
$\mathrm{10}_{145}$ & $\mathrm{10n}_{14}$ & $2$ & $[2,4]$ & $2$ \\
$\mathrm{10}_{161}$ & $\mathrm{10n}_{31}$ & $3$ & $[2,5]$ & $2$ \\
$\mathrm{11n}_{102}$ & $\mathrm{11n}_{102}$ & $2$ & $[2,4]$ & $2$ \\
$\mathrm{11n}_{104}$ & $\mathrm{11n}_{104}$ & $4$ & $[2,4]$ & $2$ \\
$\mathrm{11n}_{135}$ & $\mathrm{11n}_{135}$ & $3$ & $[2,5]$ & $2$ \\
$\mathrm{12n}_{0121}$ & $\mathrm{12n}_{0121}$ & $2$ & $[2,4]$ & $2$ \\
$\mathrm{12n}_{0233}$ & $\mathrm{12n}_{0233}$ & $4$ & $[2,4]$ & $2$ \\
$\mathrm{12n}_{0235}$ & $\mathrm{12n}_{0235}$ & $4$ & $[2,4]$ & $2$ \\
$\mathrm{12n}_{0242}$ & $\mathrm{12n}_{0242}$ & $5$ & $[2,3]$ & $2$ \\
$\mathrm{12n}_{0404}$ & $\mathrm{12n}_{0404}$ & $2$ & $[2,4]$ & $2$ \\
$\mathrm{12n}_{0474}$ & $\mathrm{12n}_{0474}$ & $4$ & $[2,4]$ & $2$ \\
$\mathrm{12n}_{0475}$ & $\mathrm{12n}_{0475}$ & $3$ & $[2,4]$ & $2$ \\
$\mathrm{12n}_{0522}$ & $\mathrm{12n}_{0522}$ & $3$ & $[2,4]$ & $2$ \\
$\mathrm{12n}_{0575}$ & $\mathrm{12n}_{0575}$ & $4$ & $[2,4]$ & $2$ \\
$\mathrm{12n}_{0581}$ & $\mathrm{12n}_{0581}$ & $3$ & $[2,4]$ & $2$ \\
$\mathrm{12n}_{0582}$ & $\mathrm{12n}_{0582}$ & $2$ & $[2,4]$ & $2$ \\
$\mathrm{12n}_{0591}$ & $\mathrm{12n}_{0591}$ & $4$ & $[2,5]$ & $2$ \\
$\mathrm{12n}_{0721}$ & $\mathrm{12n}_{0721}$ & $4$ & $[2,4]$ & $2$ \\
$\mathrm{12n}_{0725}$ & $\mathrm{12n}_{0725}$ & $5$ & $[2,3]$ & $2$ \\
$\mathrm{12n}_{0749}$ & $\mathrm{12n}_{0749}$ & $3$ & $[2,5]$ & $2$ \\
$\mathrm{12n}_{0851}$ & $\mathrm{12n}_{0851}$ & $3$ & $[2,5]$ & $2$ \\
\end{tabular}
} 
\caption{The 27 knots with newly-computed crosscap numbers}
\label{tab-new}
\end{table}

To conclude, we note that the integer programming framework given in
Algorithms~\ref{a-cc-ipexact} and \ref{a-cc-ipfloat} extends beyond
the specific problem of computing crosscap numbers---it can be used as
a general framework for locating normal surfaces with various properties.
For instance, by optimising Euler characteristic
this integer programming framework can be used to test for
0-efficiency \cite{jaco03-0-efficiency}, compute connected sum
decompositions \cite{jaco03-0-efficiency}, and perform 3-sphere
recognition \cite{burton10-quadoct}.  Given current advances in
exact integer programming solvers, this potential for applying
optimisation techniques to problems in low-dimensional topology
is only beginning to be explored.

%
%

\section*{Acknowledgements}

The authors are grateful to the
Queensland Cyber Infrastructure Foundation and
RMIT University for the use of their
high-performance computing facilities.
The first author is supported by the Australian Research Council
under the Discovery Projects funding scheme
(projects DP1094516 and DP110101104).

%
%

\small
\bibliographystyle{amsplain}
\bibliography{pure}

%
%

\bigskip
\noindent
Benjamin A.~Burton \\
School of Mathematics and Physics,
The University of Queensland \\
Brisbane QLD 4072, Australia \\
(bab@maths.uq.edu.au)

\bigskip
\noindent
Melih Ozlen \\
School of Mathematical and Geospatial Sciences,
RMIT University \\
GPO Box 2476V,
Melbourne VIC 3001, Australia \\
(melih.ozlen@rmit.edu.au)

\end{document}